\documentclass[12pt]{article}
\usepackage{a4wide}

\usepackage[a4paper, margin=2.3cm]{geometry}
\usepackage{amsmath,amssymb,amsthm}
\usepackage{color}
\usepackage{parskip}
\usepackage{hyperref}
\usepackage{parskip}
\usepackage{setspace}
\usepackage{graphicx}
\usepackage{mathtools}
\usepackage[thinc]{esdiff}
 
\newtheorem{theorem}{Theorem}[section]
\newtheorem{remark}{Remark}[section]
\newtheorem{corollary}[theorem]{Corollary}
\newtheorem{lemma}[theorem]{Lemma}

\newtheorem{definition}[theorem]{Definition}

\newtheorem*{definition*}{Definition}

\def\RR{\mathcal{R}}

\def\supp{\hbox{supp\,}}

\def\RR{\mathbb{R}}

\def\supp{\text{supp}}

\newcommand{\Comment}[1]{}

\newcommand{\cbr}[1]{\left\{ {#1} \right\}}

\theoremstyle{remark}

\begin{document}
\title{Falconer type functions in three variables}
\author{Doowon Koh\and Thang Pham\and Chun-Yen Shen}
\date{}
\maketitle

\begin{abstract}
Let $f\in \mathbb{R}[x, y, z]$ be a quadratic polynomial that depends on each variable and that does not have the form $g(h(x)+k(y)+l(z))$. Let $A, B, C$ be compact sets in $\mathbb{R}$. Suppose that $\dim_H(A)+\dim_H(B)+\dim_H(C)>2$, then we prove that the image set $f(A, B, C)$ is of positive Lebesgue measure. Our proof is based on a result due to Eswarathasan, Iosevich, and Taylor (Advances in Mathematics, 2011), and a combinatorial argument from the finite field model.
\end{abstract}

\section{Introduction}

Let $\mathbb{F}_q$ be an arbitrary finite field of order $q$, where $q$ is a prime power. A function $f\colon \mathbb{F}_q^l\to \mathbb{F}_q$ is called a moderate expander with the exponent $\epsilon$ if for all $A\subset \mathbb{F}_q$ with $|A| > q^{1-\epsilon}$, we have $|f(A, A, \ldots, A)|\gg q$, where we write $X\ll Y$ if $X\le C Y$ for some constant $C>0$ not depending on $q$.

Since 2005, there is a series of papers on seeking moderate expanders with the biggest exponents. When $l=2$, a remarkable result of Tao in \cite{tao} says that any polynomial in two variables which is not of the form $g(h(x)+k(y))$ or $g(h(x)k(y))$, for some polynomials $g, h, k$, is a moderate expander over $\mathbb{F}_q$ with the exponent $\frac{1}{16}$. When $l>2$, we know that the following polynomials are expanders with the exponent $\frac{1}{3}$: $x+yz$ \cite{shpas}, $x+(y-z)^2$ and  $x(y+z)$ \cite{vinh}, $(x-y)^2+(z-t)^2$ \cite{chap}, $xy+zt$ \cite{ha2}, $xy+z+t$ \cite{sak}. For more variables polynomials, it is expectable to have bigger exponents. For instance, $(x-y)(z-t)$ with $\epsilon=\frac{1}{3}+\frac{1}{13542}$ in \cite{mu}, $xy+(z-t)^2$ with $\epsilon=\frac{1}{3}+\frac{1}{24}$ in \cite{vinh}, many other examples can also be found in \cite{KMPV, VAN}. 

When $f$ is the algebraic distance function, namely, $f(x, y, z, t)=(x-y)^2+(z-t)^2$. A result of Iosevich and Rudnev \cite{io} on the Erd\H{o}s-Falconer distance problem tells us that for any $A\subset \mathbb{F}_q$, if $|A|\ge 2q^{1-\frac{1}{4}}$, then $f(A, A, A, A)=\mathbb{F}_q$. The exponent $\frac{1}{4}$ has been improved to $\frac{1}{3}$ by Chapman et al. \cite{chap}, and the best current exponent is $\frac{3}{8}$ over prime fields 
which follows from a recent result due to Murphy et al. \cite{mup1}. 

In the setting of prime fields, a very general result on moderate expanders in three variables has been proved by the second listed author, Vinh, and De Zeeuw in \cite{TPHAM}. More precisely, for any quadratic polynomial that depends on each variable and that does not have the form $g(h(x)+k(y)+l(z))$, and sets $A, B, C\subset \mathbb{F}_p$ with $|A|=|B|=|C|=N$, we have 
\[|f(A, B, C)\gg \min \left\lbrace  N^{3/2}, p\right\rbrace.\]
In other words any such a quadratic polynomial is a moderate expander of the exponent $\frac{1}{3}$. 

Corresponding to moderate expanders over finite fields, we have the following definition of  \textit{Falconer type functions. }
\begin{definition}
Let $\Phi\colon \mathbb{R}^l\to \mathbb{R}$, we say $\Phi$ is a Falconer type function with the threshold $\epsilon$ if for any compact set $A\subset \mathbb{R}$ of Hausdorff dimension at least $1-\epsilon$, then the image set $\Phi(A, \ldots, A)$ is of positive Lebesgue measure. 
\end{definition}

This definition is inspired by the Falconer distance conjecture, which says that for any compact set $E\subset \mathbb{R}^d$, if the Hausdorff dimension of $E$ is greater than $\frac{d}{2}$, then the Lebesgue measure of the distance set $\Delta(E)$ is positive. It is still widely open in any dimensions and the best current thresholds are $\frac{d}{2}+\frac{1}{4}$ in even dimensions due to Du, Iosevich, Ou, Wang, Zhang in \cite{DIOZ},  and $\frac{d}{2}+\frac{d}{4d-2}$ in odd dimensions by Du and Zhang in \cite{Du2}. Several extensions of this problem have been studied in the literature, for instance, see \cite{td-1, td0, td1, td2} and references therein. 

The main purpose of this paper is to provide a general family of Falconer type functions in three variables and our family of functions is in general optimal. Our proofs are based on a result due to Eswarathasan, Iosevich, and Taylor in \cite{iosevich}, and a combinatorial argument from the finite field model.
It will be seen through our approach that there is a close connection between the finite field model and the continuous setting. In the rest of this paper, we write $\dim_H(X)$ for the Hausdorff dimension of $X$, and $|X|$ for the Lebesgue measure of $X$. 

The following is our main result. 

\begin{theorem}\label{quadratic}
Let $f\in \mathbb{R}[x, y, z]$ be a quadratic polynomial that depends on each variable and that does not have the form $g(h(x)+k(y)+l(z))$. For compact sets $A, B, C\subset \mathbb{R}$ if $\dim_H(A)+\dim_H(B)+\dim_H(C)>2$, then $|f(A, B, C)|>0$. 
\end{theorem}
\begin{corollary}\label{thm1}
Let $A, B, C$ be compact sets in $\mathbb{R}$. Suppose that $\dim_H(A)+\dim_H(B)+\dim_H(C)>2$, then we have 
\begin{enumerate}
\item  $|(A-B)^2+C|>0$.
\item $|AB+C|>0$.
\item $|A(B+C)|>0$.
\end{enumerate}
\end{corollary}
We note that case (3) in Corollary \ref{thm1} was first proved by Liu in \cite{liu2} through the Mattlia integrals and using group actions. It is also not hard to see that the dimensional lower bound in Theorem \ref{quadratic} is in general sharp. For instance, we can take $f(x, y, z)=xy+z$. Let $C$ be a set in $[0, 1]$ that has Hausdorff dimension $1$, but $|C|=0$ (for example, take $C$ to be like a Cantor set), and $A=\{0\}, B=[0,1]$. Then, we have $\dim_H(A)+\dim_H(B)+\dim_H(C)=2$. and  $f(A, B, C)=C$. Hence, $|f(A, B, C)|=0$.  

The exclusion of the form of $g(h(x)+k(y)+l(z))$ is natural and necessary.  A result of Schmeling and Shmerkin \cite{SS} says that for for any real numbers $\alpha_1, \alpha_2, \alpha_3$ such that $0 \leq \alpha_1 \leq \alpha_2 \leq \alpha_3 \leq 1$, one can explicitly construct a compact set $A \subset [0, 1]$ such that  $\dim_H(A)= \alpha_1$, $\dim_H(A+A)=\alpha_2$ and $\dim_H(A+A+A)=\alpha_3$. Therefore, for $f(x, y, z)=(x+y+z)^2$ or $f(x, y, z)=x^2+y^2+z^2$, with $\alpha_1=0.7$ and $0 < \alpha_1 < \alpha_3 < 1$, we have a set $A\subset [0, 1]$ with $\dim_H(A)+\dim_H(A)+\dim_H(A)>2$ such that $|f(A, A, A)|=0$. 

When $l\ge 4$ is even, it follows from a result of Eswarathasan, Iosevich, and Taylor  in \cite[Theorem 1.8]{iosevich} that for any compact set $A\subset \mathbb{R}$ with $\dim_H(A)>\frac{l+2}{2l}$ and any smooth function $\Phi$ satisfying the Phong-Stein curvature condition, we have $|\Phi(A, \ldots, A)|>0$. Due to the generality of the function $\Phi$, this result in general does not offer the best threshold. 
For example, let $l=4$, and let $\Phi(x, y, z, t)=(x-z)^2+y\cdot t$, then we would require the condition $\dim_H(A)>3/4$.  However, if one follows our proof of Corollary \ref{thm1} in the next section, then it is not hard to see that we only need the condition $\dim_H(A)>5/8$, which is directly in the line with L. Vinh's theorem in \cite{vinh} over arbitrary finite fields. The same result also holds for the function $\Phi(x, y, z, t)=(x-z)^2+(y-t)^2$, which is a  consequence of the recent result on the Falconer distance problem in  \cite{guth, DIOZ}. A more general form of $\Phi$ in four variables can also be derived from the recent work \cite{iosevich3}. In addition, one also can discuss some related and stronger results about non-empty interiors that can be found in the papers \cite{cite1} and \cite{cite2}. 

Let $A$ be a compact set in $\mathbb{R}$, it has been proved by Iosevich and Liu in \cite{Alexliu} that if $\dim_H(A)>\frac{3}{5}$, then the distance set of $E=A^3$ is of positive Lebesgue measure. As an application of Theorem \ref{quadratic}, we have the following improvement. 
\begin{corollary}\label{distance}
Let $A$ be a compact set in $\mathbb{R}$ and $E=A\times A\times A\subset \mathbb{R}^3$. If $\dim_H(A)>\frac{4}{7}$, then we have $|\Delta(E)|>0.$ 
\end{corollary}
We note that when $E=A^d$, if we wish the distance set has non-empty interior, then from the recent work \cite{KPS}, the following condition would be enough
$$ \dim_H(A) > \left\{ \begin{array}{ll} \frac{d+1}{2d} \quad &\mbox{if}~~ 2\le d\le 4\\
\frac{d+1}{2d}-\frac{d-4}{2d(3d-4)}\quad &\mbox{if}~~ 5\le d\le 26\\
\frac{d+1}{2d}-\frac{23d-228}{114d(d-4)}\quad &\mbox{if}~~ 27\le d. \end{array} \right.$$
We also remark that, in the spirit of Theorem \ref{quadratic}, it would be very interesting to study the Hausdorff dimension of the set $f(A, B, C)$ for any quadratic polynomial that does not have the form $g(h(x)+k(y)+l(z))$. It is worth mentioning a result in this direction by Bourgain in \cite{bg1, bg2}, which states that for any $A, B\subset \mathbb{R}$ with $\dim_H(A)=\dim_H(B)\in (0, 1)$, and $C\subset \mathbb{R}$ with $\dim_H(C)\ge k>0$, then $\dim_H(A+\theta B)\ge \epsilon$, for some $\theta\in C$, where $\epsilon>0$ depends only on $\dim_H(A)$ and $\dim_H(B)$. The recent progress on this problem can be found in \cite{o} by Orponen. We hope to address this question in a subsequent paper.

\section{Some special cases of Theorem \ref{quadratic}}
In this section, we provide a proof of Corollary \ref{thm1}, which will provide some intiuitive ideas behind the proof of Theorem \ref{quadratic}.

Given $f\colon \mathbb{R}^d\to \mathbb{R}$ and $t\in \mathbb{R}$, define 
\[T_{\Psi_t}f(\mathbf{x}):=\int_{\{\Psi(\mathbf{x}, \mathbf{y})=t\}}f(\mathbf{y})\psi(\mathbf{x}, \mathbf{y})d\sigma_{\mathbf{x}, t}(\mathbf{y}),\]
where $d\sigma_{\mathbf{x}, t}$ is the Lebesgue measure on the set $\{\mathbf{y}\colon \Psi(\mathbf{x}, \mathbf{y})=t\}$ and $\psi$ is a smooth cut-off function and $ \Psi(\mathbf{x}, \mathbf{y}) : \mathbb R^d \times \mathbb R^d \rightarrow \mathbb R$ are smooth functions with some suitable assumptions. Let $L^2_s(\mathbb{R}^d)$ denotes the usual $L^2$-Sobolev space of $L^2$ functions with $s$ generalized derivatives in $L^2(\mathbb{R}^d)$.
To prove Theorem \ref{thm1}, let us first recall the following result by Eswarathasan, Iosevich and Taylor \cite{iosevich}.
\begin{theorem} [Proposition 2.2 \cite{iosevich}]\label{2.1.1}
Let $\mu$ be a probability measure on a compact set $E \subset \mathbb R^d$ and assume $T_{\Psi_t}$ maps $L^2$ to $L^2_s$ with $ d-s < \alpha < d$ with constants uniform in a small neighborhood of $t$, where $\alpha=\dim E$.

Then
$$\mu \times \mu \{(\mathbf{x}, \mathbf{y}) \in E \times E : \epsilon \leq |\Psi(\mathbf{x},\mathbf{y})| \leq t+ \epsilon\} \lesssim \epsilon.$$

\end{theorem}

\begin{remark}
It has been mentioned or can be directly checked in the paper \cite{iosevich} that the same result holds for $E \times F$, i.e
$$\mu_E \times \mu_F \{(\mathbf{x}, \mathbf{y}) \in E \times F : \epsilon \leq |\Psi(\mathbf{x},\mathbf{y})| \leq t+\epsilon\} \lesssim \epsilon,$$
if  $T_{\Psi_t}$ maps $L^2$ to $L^2_s$ with $ d-s < \alpha, \beta < d$, where $\alpha=\dim E$ and $\beta=\dim F$. 
\end{remark}

To deal with the boundedness of the Radon transforms $T_{\Psi_t}$, the curvature condition can be checked partially by a celebrated result of Phong-Stein \cite{phong} that if a smooth function $\Psi : U \times V \subset \mathbb R^d \times \mathbb R^d \rightarrow \mathbb R$ satisfies the so-called Phong-Stein rotational curvature condition that
$$\det \begin{pmatrix}
0 & \nabla_\mathbf{x}\Psi \\
-\nabla_\mathbf{y}\Psi & \frac{\partial^2 \Psi}{\partial x_i \partial y_j} 
\end{pmatrix} \ne 0,$$ on the set $\{(\mathbf{x}, \mathbf{y})\colon \Psi(\mathbf{x},\mathbf{y})=t\}$, then the operator $T_{\Psi_t}$ is uniformly bounded from $L^2$ to $L^2_s$ on a small neighborhood of $t$ with $s= \frac{d-1}{2}$.

Let $\nu$ be a measure supported on the range of $\Phi$ defined by 
\[\int f(t)\nu(t)=\int f(\Phi(x, y, z))d\mu_A(x)d\mu_B(y)d\mu_C(z),\]
where $\mu_X$ denotes a probability measure on $X$ satisfying the Frostman condition.  Our goal now is to show that the $L^2$ norm of $\nu$ is finite so that it immediately implies that the Lebesgue measure of the support of $\Phi$ is positive.  
We begin with an approximation of  identity for $\nu$ as follows. We choose $\phi \in C_0^{\infty}(\RR)$ with $\phi \geq 0$, $\supp(\phi) \subseteq B(0, 1)$, and $\int \phi(x) dx = 1$, and the associated approximate identity is $\phi_{\epsilon}(x)=\epsilon^{-1}\phi(\epsilon^{-1}x)$ for $\epsilon > 0$. 

Since $\widehat{\phi_{\epsilon} \ast \nu} \to \widehat{\nu}$, it suffices to show that
\begin{align}\label{nu to nug approx 1}
\int_{\RR} (\phi_{\epsilon} \ast \nu)^2(t) dt \lesssim  1.
\end{align}

For $t \in \RR$, we have 
\begin{align*}
\phi_{\epsilon} \ast \nu(t) 
&=  \int_{\RR^3} \phi_{\epsilon}\left(t-\Phi(x, y, z)\right) d\mu_A(x) d\mu_B(y)d\mu_C(z) \\
&\lesssim  \int_{\RR^3} \epsilon^{-1} \chi_{\cbr{\bigg|t-\Phi(x, y, z)\bigg| \leq \epsilon}} d\mu_A(x) d\mu_B(y)d\mu_C(z),
\end{align*}
where $\chi_S$ denotes the indicator function of a set $S$. 
Hence, applying the triangle inequality, 
\begin{gather*}
\int_{\RR} (\phi_{\epsilon} \ast \nu)^2(t) dt 
\\
\lesssim 
\epsilon^{-2} \int
\chi_{\cbr{\left|t-\Phi(x, y, z)\right| \leq \epsilon}}
\chi_{\cbr{\left|t-\Phi(x', y', z')\right| \leq \epsilon}}
d\mu_A(x) d\mu_B(y) d\mu_C(z)  d\mu_A(x') d\mu_B(y') d\mu_C(z') dt 
\\
\leq 
\epsilon^{-2} \int  
\chi_{\cbr{\left|t-\Phi(x, y, z)|\right| \leq \epsilon}}
\chi_{\cbr{\left|\Phi(x, y, z)-\Phi(x', y', z')\right| \leq 2\epsilon}}
d\mu_A(x) d\mu_B(y) d\mu_C(z)  d\mu_A(x') d\mu_B(y') d\mu_C(z') dt \\
\le \epsilon^{-1} \int  \chi_{\cbr{\left|\Phi(x, y, z)-\Phi(x', y', z')\right| \leq 2\epsilon}}
d\mu_A(x) d\mu_B(y) d\mu_C(z)  d\mu_A(x') d\mu_B(y') d\mu_C(z').
\end{gather*}

{\bf Case $1$:} $(A-B)^2+C$, $\Phi(x, y, z)=(x-y)^2+z$.

Set 
\[\Psi(\mathbf{u}, \mathbf{v}):=(u_1-v_1)^2-(u_2-v_2)^2+u_3-v_3\]
being a smooth function from $\mathbb{R}^3\times \mathbb{R}^3\to \mathbb{R}$, and set $U= A \times B \times C$, $V= B \times A \times C$

Set $\mu_U=\mu_A\times \mu_B\times \mu_C$ and $\mu_V=\mu_B\times \mu_A\times \mu_C$. 

With $G(x, y', z, y, x', z')=\chi_{\cbr{\left|\Phi(x, y, z)-\Phi(x', y', z')\right| \leq 2\epsilon}}$, we have
\begin{gather*}
\int_{\RR} (\phi_{\epsilon} \ast \nu)^2(t) dt 
\\
\le \epsilon^{-1} \int  \chi_{\cbr{\left|\Phi(x, y, z)-\Phi(x', y', z')\right| \leq 2\epsilon}}
d\mu_A(x) d\mu_B(y) d\mu_C(z)  d\mu_A(x') d\mu_B(y') d\mu_C(z')\\
=\epsilon^{-1}\int G(x, y', z, y, x', z')d\mu_A(x)d\mu_B(y')d\mu_C(z)d\mu_B(y)d\mu_A(x')d\mu_C(z')\\
= \epsilon^{-1}\int \chi_{\cbr{|\Psi(\mathbf{u}, \mathbf{v})|\le 2\epsilon}}d\mu_U(\mathbf{u})d\mu_V(\mathbf{v})\\
=\epsilon^{-1}\cdot \mu_U\times \mu_V\left\lbrace (\mathbf{u}, \mathbf{v})\colon |\Psi(\mathbf{u}, \mathbf{v})|\le 2\epsilon \right\rbrace.
\end{gather*}

From here we can apply Theorem \ref{2.1.1}. By a direct computation, we have 

\[\det \begin{pmatrix}
0&\nabla_\mathbf{u} \Psi\\
-(\nabla_\mathbf{v}\Psi)^T&\frac{\partial^2\Psi}{\partial u_i\partial v_j}
\end{pmatrix}=\det \begin{pmatrix}
0&2u_1-2v_1&-2u_2+2v_2&1\\
2v_1-2u_1&-2&0&0\\
-2v_2+2u_2&0&2&0\\
1&0&0&0
\end{pmatrix}    =4. \]
Hence, the Phong-Stein condition is satisfied, and we have $s=2$ since $\dim_H(U), \dim_H(V)>2$ for our assumption that $\dim_H(A)+\dim_H(B)+ \dim_H(C)>2$. 

{\bf Case $2$:} $AB+C$, $\Phi(x, y, z)=xy+z$.

The proof is similar to the previous case with 

\[\Psi(\mathbf{u}, \mathbf{v}):=u_1v_1+u_3-u_2v_2-v_3,\]
and set $U= A \times B \times C$, $V= B \times A \times C$.

{\bf Case $3$:} $A(B+C)$, $\Phi(x, y, z)=x(y+z)$.

The proof of this case is quite different compared to the above two cases. More precisely, let\[U=\{(x_1, y_2, x_1\cdot x_3)\colon x_1\in A, y_2\in B, x_3\in C\}, ~V=\{(y_1, x_2, x_2\cdot y_3)\colon y_1\in B, x_2\in A, y_3\in C\}.\] Define 
\[\Psi(\mathbf{u}, \mathbf{v}):=u_1v_1-u_2v_2+v_3-u_3.\]
With $G(x, y', z, x', y, z')=\chi_{\cbr{\left|\Phi(x, y, z)-\Phi(x', y', z')\right| \leq 2\epsilon}}$, we have
\begin{gather*}
\int_{\RR} (\phi_{\epsilon} \ast \nu)^2(t) dt 
\\
\le \epsilon^{-1} \int  \chi_{\cbr{\left|\Phi(x, y, z)-\Phi(x', y', z')\right| \leq 2\epsilon}}
d\mu_A(x) d\mu_B(y) d\mu_C(z)  d\mu_A(x') d\mu_B(y') d\mu_C(z')\\
=\epsilon^{-1}\int G(x, y', z, x', y, z')d\mu_A(x) d\mu_B(y) d\mu_C(z)  d\mu_A(x') d\mu_B(y') d\mu_C(z')\\
=\epsilon^{-1}\int G(x, y', z, x', y, z')d\mu_A(x) d\mu_B(y') d\mu_C(z)  d\mu_A(x') d\mu_B(y) d\mu_C(z').
\end{gather*}
Without loss of generality, we assume that $A$ is a compact subset of $[0, 1]$ and has positive distance from $0$. To see this, we let $A_+=A\cap [0, 1]$ and $A_-=A\cap [-1, 0]$. It is clear that one of them must have the same dimension as $A$. If it is $A_-$, then we will work with the polynomial $-x(y+z)$ instead of $x(y+z)$, which does not affect the Monge-Ampere determinant for our purpose. Thus, we can assume that $A\subset [0, 1]$. Note that from our assumption that $\dim_H A + \dim_H B + \dim_H C >2$, we have $\dim_H A > 0$. In other words, $A$ cannot be only supported on countable points. We now consider the restriction of $\mu_A$ into $[0, 1/2]$ and $[1/2, 1]$. If $\mu([1/2, 1])>0$, then we set $A:=[1/2, 1]$. Otherwise, we continue to consider the intervals $[0, 1/4]$ and $[1/4, 1/2]$. We repeat this process until there is an interval $[\frac{1}{2^N}, \frac{1}{2^{N-1}}]$ with positive measure. This holds because otherwise $\mu_A(A)=\lim_{N\to \infty}\mu_A \left([0, \frac{1}{2^N}]\right)=0$.

We now consider the maps $F_1, F_2$:
\[F_1\colon \mathbb{R}^3\to \mathbb{R}^3, ~(x, y', z)\to (x, y', x\cdot z), (A\times B\times C)\to U,\]
\[F_2\colon \mathbb{R}^3\to \mathbb{R}^3, ~(x', y, z')\to (y, x', x'\cdot z'), (A\times B\times C)\to V.\]
It is not hard to check that under our assumptions on $A, B, C$, these maps are bi-Lipschitz over $A\times B\times C$. 

We also have 
\begin{align*}
&\int G(x, y', z, x', y, z')d\mu_A(x) d\mu_B(y) d\mu_A(x) d\mu_B(y') d\mu_C(z)  d\mu_A(x') d\mu_B(y) d\mu_C(z')\\
&=\int G(F_1^{-1}(\mathbf{u}), F_2^{-1}(\mathbf{v}))d(F_1)_{*}\mu_A\times \mu_B\times \mu_C(\mathbf{u})d(F_2)_{*}\mu_A\times \mu_B\times \mu_C(\mathbf{v})\\
&=\int \chi_{\cbr{|\Psi(\mathbf{u}, \mathbf{v})|\le 2\epsilon}}d(F_1)_{*}\mu_A\times \mu_B\times \mu_C(\mathbf{u})d(F_2)_{*}\mu_A\times \mu_B\times \mu_C(\mathbf{v}).
\end{align*}

We now show that $(F_i)_*\mu_A\times \mu_B\times \mu_C$ are Frostman measures. 
Indeed,
\begin{align*}
&(F_i)_*\mu_A\times \mu_B\times \mu_C(B(\mathbf{x}, r))=\mu(F_i^{-1}B(\mathbf{x}, r))=\mu \left(  \left\lbrace  F_i^{-1}(\mathbf{x}')\colon |\mathbf{x}-\mathbf{x}'|\le r  \right\rbrace   \right).
\end{align*}

Therefore, we can use the bi-Lipschitz property of the maps $F_1$ and $F_2$ to imply that 
\begin{align*}
&(F_i)_*\mu_A\times \mu_B\times \mu_C(B(\mathbf{x}, r))=\mu_A\times _B\times\mu_C(F_i^{-1}B(\mathbf{x}, r))\\&=\mu_A\times \mu_B\times \mu_C \left(  \left\lbrace  F_i^{-1}(\mathbf{x}')\colon |\mathbf{x}-\mathbf{x}'|\le r  \right\rbrace   \right)\\
&\le \mu_A\times \mu_B\times \mu_C \left( B(F_i^{-1}(\mathbf{x}), cr)\right),
\end{align*}
for some uniform constant $c$. Since  $\mu_A\times \mu_B\times \mu_C$ is a Frostman measure, the measures $(F_i)_*\mu_A\times \mu_B\times \mu_C$ are also Frostman. 

To apply Theorem \ref{2.1.1}, we need to check if $\dim_H(U), \dim_H(V)>2$ and the curvature condition. 

The dimension conditions follow from the fact that $F_1$ and $F_2$ are bi-Lipschitz over $A\times B\times C$. For the curvature condition, we have


\begin{equation}\label{ewqq}\det \begin{pmatrix}
0&\nabla_\mathbf{u} \Psi\\
-(\nabla_\mathbf{v}\Psi)^T&\frac{\partial^2\Psi}{\partial u_i\partial v_j}
\end{pmatrix}=\det \begin{pmatrix}
0&v_1&-v_2&1\\
u_1&1&0&0\\
-u_2&0&-1&0\\
1&0&0&0
\end{pmatrix}    =1. \end{equation}
Thus, the Phong-Stein condition is satisfied. $\square$

\begin{remark}\label{AGR}
We remark that our results can be extended to a more general setting, precisely, define 
\[\Psi(\mathbf{u}, \mathbf{v})=\Phi(x, y, z)-\Phi(x', y', z'),\]
where $(\mathbf{u}, \mathbf{v})$ is a certain permutation of $(x, y, z, x', y', z')$ that may depend on whether $\Phi$ is symmetric such that 
\begin{enumerate}
\item $T_{\Psi_t}\colon L^2(\mathbb{R}^d)\to L^2_s(\mathbb{R}^d)$ with constants uniform in $t$ in a small neighborhood of $0$, 
\item The corresponding Hausdorff dimension to variables $\mathbf{u}$ and $\mathbf{v}$ is bigger than $3-s$, 
\end{enumerate}
then $|\Phi(A, B, C)|>0$. 

For example, we can choose $(\mathbf{u}, \mathbf{v})=(x, y', z, y, x', z')$, then Corollary \ref{thm1} (1) and (2) are recovered. In addition, if we assume a stronger assumption that $\dim_H(A) > \frac{2}{3}$, $\dim_H(B) > \frac{2}{3}$ and $\dim_H(C) > \frac{2}{3}$, then the choices of permutations are more flexible.
\end{remark}

\section{Proof of Theorem \ref{quadratic}}
Some of our ideas to prove Theorem \ref{quadratic} are motivated by the results in \cite{TPHAM} that are the discrete version of expanding polynomials in finite fields. While the key ingredient in the discrete version is the Rudnev point-plane incidence bound \cite{R} which is unavailable to our continuous setting. Instead, we will still use the Phong-Stein curvature condition to construct an appropriate map $\Psi$. Moreover due to the complexity of our general assumption on the polynomials $f$, it may happen that the Phong-Stein curvature condition is not satisfied on some subsets. However after carefully analysing these bad subsets, we are able to show that these bad sets can be removed from our underlying sets so that the remaining good subsets still occupy a large portion that makes the proofs still go through. We now turn to the proofs.  

Let $f(x, y, z)$ be a quadratic polynomial that is not of the form $g(h(x)+k(y)+l(z))$. Then Theorem \ref{quadratic} is a combination of the following two lemmas. 
\begin{lemma}\label{xqq91}
Let $f(x, y, z)=axy+bxz+r(x)+s(y)+t(z)$ with $r, s, t$ are polynomials in one variable of degree at most two. Suppose $a\ne 0$ and for compact sets $A, B, C\subset \mathbb{R}$ with $\dim_H(A)+\dim_H(B)+\dim_H(C)>2$, we have $|f(A, B, C)|>0$. 
\end{lemma}
\begin{proof}
As in the proof of Corollary \ref{thm1} (3), we may assume that $A, B, C\subset (0, \infty)$ and have positive distance from $0$. The reason that if one of the sets $A, B, C$ is negative, we put a negative sign in the corresponding variables, and work with the set $ (-1)A \times (-1)B \times (-1)C$ depending on whether we put a negative sign in the variables. Thus the image of $f$ on $A \times B \times C$ is equal to the image of $f((-1)x, (-1)y, (-1)z)$ on $(-1)A \times (-1)B \times (-1)C$. In addition, our proofs only use the assumption that $\dim_H(A)+\dim_H(B)+\dim_H(C) > 2$. 

Set 
\[U=\left\lbrace  (x, y', bxz+r(x)+t(z)-s(y'))\colon (x, y', z)\in A\times B\times C  \right\rbrace,\]
and 
\[V=\left\lbrace  (ay, ax', bx'z'+r(x')+t(z')-s(y))\colon (x', y, z')\in A\times B\times C    \right\rbrace,\]
and define 
\[\Psi(\mathbf{u}, \mathbf{v})=u_1v_1-u_2v_2+u_3-v_3.\]

Let $F_1$ and $F_2$ be maps defined as follows
\[F_1\colon \mathbb{R}^3\to \mathbb{R}^3, ~(x, y', z)\to (x, y', bxz+r(x)+t(z)-s(y')),\]
\[F_2\colon \mathbb{R}^3\to \mathbb{R}^3, ~(x', y, z')\to (ay, ax', bx'z'+r(x')+t(z')-s(y)).\]

As in the proof of Corollary \ref{thm1} (3), we need to show that $\Psi$ satisfies the Phong-Stein condition and the following estimate

 \begin{equation}\label{eqx90}\int_{\RR} (\phi_{\epsilon} \ast \nu)^2(t) dt \lesssim \epsilon^{-1} \int \chi_{\cbr{|\Psi(\mathbf{u}, \mathbf{v})|\le 2\epsilon}}d(F_1)_{*}\mu_A\times \mu_B\times \mu_C(\mathbf{u})d(F_2)_{*}\mu_A\times \mu_B\times \mu_C(\mathbf{v})\lesssim 1.\end{equation}

As in the previous section, the function $\Psi$ satisfies the Phong-Stein condition. In the rest of the proof, we focus on showing the estimate (\ref{eqx90}). 

We note that $F_1$ might not be an injective map since for each $(x, y', u)\in F_1(A, B, C)$, the equation $bxz+r(x)+t(z)-s(y')=u$ can have two solutions. The same happens for $F_2$.

With $G(x, y', z, x', y, z')=\chi_{\cbr{\left|f(x, y, z)-f(x', y', z')\right| \leq 2\epsilon}}$, we have
\begin{gather*}
\int_{\RR} (\phi_{\epsilon} \ast \nu)^2(t) dt 
\\
\le \epsilon^{-1} \int  \chi_{\cbr{\left| f(x, y, z)- f(x', y', z')\right| \leq 2\epsilon}}
d\mu_A(x) d\mu_B(y) d\mu_C(z)  d\mu_A(x') d\mu_B(y') d\mu_C(z')\\
=\epsilon^{-1}\int G(x, y', z, x', y, z')d\mu_A(x) d\mu_B(y) d\mu_C(z)  d\mu_A(x') d\mu_B(y') d\mu_C(z'),\\
=\epsilon^{-1}\int G(x, y', z, x', y, z')d\mu_A(x) d\mu_B(y') d\mu_C(z)  d\mu_A(x') d\mu_B(y) d\mu_C(z').
\end{gather*}

We now partition the set $\{(x, y', z)\in A\times B\times C\}$ into two sets $X_1$ and $Y_1$ such that $F_1(X_1)=F_1(Y_1)=U$ and $F_1$ is bi-Lipschitz over each set. In other words, for each triple $(x, y', u)\in F_1(A, B, C)$, and if the equation $bxz+r(x)+t(z)-s(y')=u$ has two solutions $c_1, c_2$, then we assign $(x, y', c_1)$ to $X_1$ and $(x, y', c_2)$ to $Y_1$. If the equation only has one solution $c$, we assign $(x, y', c)$ to both sets $X_1$ and $Y_1$.
Note that $X_1$ and $Y_1$ do not need to be disjoint. Then we have $\dim_H(U)=\dim_H(X_1)=\dim_H(Y_1)$. 

On the other hand, we have either $\dim_H(X_1)$ or $\dim_H(Y_1)$ is the same as $\dim_H(A\times B\times C)$. This means that 
\[\dim_H(U)=\dim_H(X_1)=\dim_H(Y_1)=\dim_H(A\times B\times C).\]

We do the same for $\{(x', y, z')\in A\times B\times C\}=X_2\cup Y_2$. 

Therefore 
\begin{gather*}
\int G(x, y', z, x', y, z')d\mu_A(x) d\mu_B(y') d\mu_C(z)  d\mu_A(x') d\mu_B(y) d\mu_C(z')\\
\le \sum_{i, j}\int G(x, y', z, x', y, z')d\mu_A\times \mu_B\times \mu_C{\big|}_{X_i}(x, y', z)d\mu_A\times \mu_B\times \mu_C{\big|}_{Y_i}(x', y, z')\\
=\sum_{i, j}\int G(F_1^{-1}(\mathbf{u}), F_2^{-1}(\mathbf{v}))d(F_1)_{*}\mu_A\times \mu_B\times \mu_C{\big|}_{X_i}(\mathbf{u})d(F_2)_{*}\mu_A\times \mu_B\times \mu_C{\big|}_{Y_j}(\mathbf{v})\\
=\sum_{i, j}\int \chi_{\cbr{|\Psi(\mathbf{u}, \mathbf{v})|\le 2\epsilon}}d(F_1)_{*}\mu_A\times \mu_B\times \mu_C{\big|}_{X_i}(\mathbf{u})d(F_2)_{*}\mu_A\times \mu_B\times \mu_C{\big|}_{Y_j}(\mathbf{v})\lesssim \epsilon,
\end{gather*}
where we applied Theorem \ref{2.1.1} with the facts that $d(F_i)_{*}\mu_A\times \mu_B\times \mu_C{\big|}_{X_j}$ and $d(F_i)_{*}\mu_A\times \mu_B\times \mu_C{\big|}_{Y_j}$ are Frostman measures for $1\le i, j\le 2$. 
This completes the proof of the lemma.
\end{proof}

\begin{lemma}
Let $f(x, y, z)=axy+bxz+cyz+r(x)+s(y)+t(z)$ with $r, s, t$ polynomials in one variable of degree at most two. Suppose $a, b, c\ne 0$ and $f$ is not of the form $g(h(x)+k(y)+l(z))$. For compact sets $A, B, C\subset \mathbb{R}$ with $\dim_H(A)+\dim_H(B)+\dim_H(C)>2$, then $|f(A, B, C)|>0$. 
\end{lemma}
\begin{proof}
As in the previous lemma, we can assume that $A, B, C\subset (0, \infty)$, and have positive distance from $0$.

We first write $f$ as 
\[f(x, y, z)=axy+bxz+cyz+dx^2+ey^2+gz^2+hx+iy+jz,\]
where $a,b,c\ne 0$ and $d,e,g,h, i, j\in \mathbb{R}.$ 
Define
\[U=\left\lbrace (x, ay'+bz', dx^2-e(y')^2-cy'z'-g(z')^2+hx-iy'-jz')\colon x\in A, y'\in B, z'\in C  \right\rbrace,\]
\[V=\left\lbrace (ay+bz, x', d(x')^2-ey^2-cyz-gz^2+hx'-iy-jz)\colon x'\in A, y\in B, z\in C       \right\rbrace,\]
and 
\[\Psi(\mathbf{u}, \mathbf{v})=u_1v_1-u_2v_2+u_3-v_3.\]


We now adapt a combinatorial argument from \cite[Lemma 2.3]{TPHAM}.

From our assumption, we know that $f$ is not of the form $g(h(x)+k(y)+l(z))$, thus at least one of the following equations does not hold
\begin{equation*}
4de=a^2,~ 4dg =b^2,~ 4eg =c^2,~  hc=ja=ib.
\end{equation*}
If not, one can write
$$f =\left(\sqrt{d} x + \sqrt{e}y + \sqrt{g}z + \frac{h}{2\sqrt{d}}\right)^2 -\frac{h^2}{4d},$$ 
if $d, e, g$ are all squares in $\mathbb{R}$. If  all of $d,e,g$ are not squares in $\mathbb{R}$, the polynomial $f$ can be presented as
$$f =\frac{1}{deg} \left( d\sqrt{eg} x+ e\sqrt{dg}y + g\sqrt{de}z + \frac{h\sqrt{eg}}{2}\right)^2-\frac{h^2}{4d},$$
since $de, dg, eg$ are squares in  $\mathbb{R},$ and $e, d, g\ne 0$. 

By permuting the variables if necessary, in the rest of the proof, we suppose that one of the following equations is not satisfied
\begin{equation*} ib=ja, ~4eg=c^2 \end{equation*}

To study the dimensions of  $U$ and $V$, let $F_1$ and $F_2$ be maps defined as follows
\[F_1\colon \mathbb{R}^3\to \mathbb{R}^3, ~(x, y', z')\to (x, ay'+bz', dx^2-e(y')^2-cy'z'-g(z')^2+hx-iy'-jz'),\]

\[F_2\colon \mathbb{R}^3\to \mathbb{R}^3, ~(x', y, z)\to (ay+bz, x', d(x')^2-ey^2-cyz-gz^2+hx'-iy-jz).\]

The argument is the same for both $F_1$ and $F_2$, so we only prove it for $F_1$. 

We now fall into two cases:

{\bf Case $1$:} Assume that $bc-2ag\ne 0$. Define
\[S=\left\lbrace (y, z)\colon  y\in B, z\in C, ~ay+bz=-\frac{ib^2-jab}{bc-2ag}\right\rbrace.\]
For $(u, v, w)\in F_1(A, B, C)$, we consider the following equations
\[u=x,~~~ v = ay'+bz', ~~~w = dx^2 - e (y')^2-cy'z'- g(z')^2 +hx-iy'-jz'.\]
This implies
\[w = d u^2 - e (y')^2 -cy' \left(\frac{v-ay'}{b}\right) - g\left(\frac{v-ay'}{b}\right)^2 
+hu-iy'-j\left(\frac{v-ay'}{b}\right), \]
or
\[\left(b^2e-abc+a^2g\right)(y')^2  +\left(bcv - 2agv +ib^2 - jab\right)y' +\left(b^2w - b^2d u^2+g v^2 -b^2hu +bjv\right)= 0.\]

If either $b^2e-abc+a^2g$ or $bcv-2agv+ib^2-jab$ is non-zero, then there are at most two solutions for $y'$, and $z'$ is determined uniquely in terms of $y'$.

If both $b^2e-abc+a^2g$ and $bcv-2agv+ib^2-jab$ are zero, then we obtain 
\begin{equation}\label{0911}
b^2e-abc+a^2g=0, ~~~
(bc - 2ag)v +(ib - ja)b=0,~~~
b^2w - b^2d u^2+g v^2 -b^2hu +bjv=0.\end{equation}
Since $bc-2ag\ne 0$, this implies that
\[  v = -(ib^2-jab)/(bc-2ag). \]
We can avoid this case by removing $S$ from $B\times C$. 

It is enough to show that $f\bigg( A\times (B\times C\setminus S) \bigg)$ is of positive Lebesgue measure. 

Since $S$ can be a line or an empty-set, so the dimension of $B\times C\setminus S$ is the same as the dimension of $B\times C$, which is greater than $1$. 
 
To study $f\bigg( A\times (B\times C\setminus S) \bigg)$, we need to modify the definitions of $U$ and $V$, namely, we define
\[U'=\left\lbrace (x, ay'+bz', dx^2-e(y')^2-cy'z'-g(z')^2+hx-iy'-jz')\colon x\in A, (y', z')\in B\times C\setminus \mathcal{S}  \right\rbrace,\]
\[V'=\left\lbrace (ay+bz, x', d(x')^2-ey^2-cyz-gz^2+hx'-iy-jz)\colon x'\in A, (y, z)\in B\times C\setminus \mathcal{S} \right\rbrace.\]

With $G(x, y', z', x', y, z)=\chi_{\cbr{\left|f(x, y, z)-f(x', y', z')\right| \leq 2\epsilon}}$, we have
\begin{gather*}
\int_{\RR} (\phi_{\epsilon} \ast \nu)^2(t) dt 
\\
\le \epsilon^{-1} \int  \chi_{\cbr{\left|f(x, y, z)-f(x', y', z')\right| \leq 2\epsilon}}
d\mu_A(x) d\mu_{B\times C\setminus S}(y, z)   d\mu_A(x') d\mu_{B\times C\setminus S}(y', z')\\
=\epsilon^{-1}\int G(x, y', z', x', y, z)d\mu_A(x) d\mu_{B\times C\setminus S}(y', z')   d\mu_A(x') d\mu_{B\times C\setminus S}(y, z).\\
\end{gather*}
We now partition the set $\{(x, y, z)\in A\times (B\times C\setminus S)\}$ into two sets $X_1$ and $Y_1$ such that $F_1(X_1)=F_1(Y_1)=U'$ and $F_1$ is bi-Lipschitz over each set. Note that $X_1$ and $Y_1$ do not need to be disjoint. Then we have $\dim_H(U')=\dim_H(X_1)=\dim_H(Y_1)$. 

On the other hand, we have either $\dim_H(X_1)$ or $\dim_H(Y_1)$ is the same as $\dim_H(A\times (B\times C\setminus S))$. This means that 
\[\dim_H(U')=\dim_H(X_1)=\dim_H(Y_1)=\dim_H(A\times (B\times C\setminus S)).\]

We do the same for $\{(x', y, z)\in A\times (B\times C\setminus S)\}=X_2\cup Y_2$. 

So  
\begin{gather*}
\int G(x, y', z', x', y, z)d\mu_A(x) d\mu_{B\times C\setminus S}(y', z') d\mu_A(x') d\mu_{B\times C\setminus S}(y, z) \\
\le \sum_{i, j}\int G(x, y', z', x', y, z)d\mu_A\times d\mu_{B\times C\setminus S}{\big|}_{X_i}(x, y', z')d\mu_A\times d\mu_{B\times C\setminus S}{\big|}_{Y_i}(x', y, z)\\
=\sum_{i, j}\int G(F_1^{-1}(\mathbf{u}), F_2^{-1}(\mathbf{v}))d(F_1)_{*}\mu_A\times \mu_{B\times C\setminus S}{\big|}_{X_i}(\mathbf{u})d(F_2)_{*}\mu_A\times \mu_{B\times C\setminus S}{\big|}_{Y_j}(\mathbf{v})\\
=\sum_{i, j}\int \chi_{\cbr{|\Psi(\mathbf{u}, \mathbf{v})|\le 2\epsilon}}d(F_1)_{*}\mu_A\times \mu_{B\times C\setminus S}{\big|}_{X_i}(\mathbf{u})d(F_2)_{*}\mu_A\times \mu_{B\times C\setminus S}{\big|}_{Y_j}(\mathbf{v})\lesssim \epsilon,
\end{gather*}
where we applied Theorem \ref{2.1.1} with the facts that $d(F_i)_{*}\mu_A\times \mu_{B\times C\setminus S}{\big|}_{X_j}$ and $d(F_i)_{*}\mu_A\times \mu_{B\times C\setminus S}{\big|}_{Y_j}$ are Frostman measures for $1\le i, j\le 2$.

In other words, $F_1$ is bijective on the set $A\times (B\times C\setminus S)$. 

{\bf Case $2$:} Assume that $bc-2ag=0$. There is only a difference compared to the previous when both $b^2e-abc+a^2g$ and $bcv-2agv+ib^2-jab$ are zero. Then, we obtain 
\begin{equation}\label{0911}
b^2e-abc+a^2g=0, ~~~
(bc - 2ag)v +(ib - ja)b=0,~~~
b^2w - b^2d u^2+g v^2 -b^2hu +bjv=0.\end{equation}
From these equations, one can check that $ib=ja$ and $4eg=c^2$. This contradicts to our assumption. 

Thus, in this case, we do not have to remove $S$ from $A\times B\times C$. The rest of the proof is the same as above. 
\end{proof}

\section{Proof of Corollary \ref{distance}}
To prove Corollary \ref{distance}, we first recall the following result from Liu in \cite{liu}.

\begin{lemma}[\cite{liu}]\label{lm2'}
For $E\subset \mathbb{R}^2$ with $\dim_H(E)>1$, then we have 
\[\dim_H(\Delta(E))\ge \min \left\{\frac{4}{3}\dim_H(E)-\frac{2}{3}, ~1 \right\}.\]
\end{lemma}

\begin{proof}[Proof of Corollary \ref{distance}]
Set $A=B$ and $C=\Delta(A^2)^2$. It has been proved in \cite[Lemma 2.3]{KPS} that $\dim_H(C)=\dim_H(\Delta(A^2))$. Using Lemma \ref{lm2'} and Corollary \ref{thm1} (1), we can conclude that $|\Delta(A^3)^2|>0$ whenever $\dim_H(A)>4/7$. This implies that $|\Delta(A^3)|>0$ under $\dim_H(A)>4/7$.
\end{proof}

Finally, it is also worth to mention that under the assumptions in our Theorem 1.2, it seems not possible to show that the image of $f(x,y,z)$ contains an interval. Although we are not aware of any example in the continuous setting, a result of Bienvenu, Hennecart and Shkredov in the setting of finite fields $\mathbb F_p$ \cite{BHS} shows that if we take $f(x,y,z)=x(y+z)$, then there exists a set $A \subset \mathbb F_p$ with $|A| > cp$ such that $A(A+A)$ does not contain the whole field $\mathbb F_p$.

\section*{Acknowledgments}
The authors would like to thank Prof. Allan Greenleaf for informing us about a mistake in the statement of Theorem 1.2 in the earlier version.  Theorem 1.2 from that version is now replaced by Remark \ref{AGR} of this paper. 

Doowon Koh was supported by the National Research Foundation of Korea (NRF) grant funded by the Korea government (MIST) (No. NRF-2018R1D1A1B07044469). Thang Pham was supported by Swiss National Science Foundation grant P4P4P2-191067. Chun-Yen Shen was supported in part by MOST, through grant 108-2628-M-002-010-MY4.

 \bibliographystyle{amsplain}

\end{document}